\documentclass[]{amsart}
\usepackage{amssymb,amsfonts}
\usepackage{enumerate}
\usepackage{mathrsfs}
\usepackage{url}
\usepackage{graphicx}
\usepackage{float}

\newtheorem{thm}{Theorem}[section]
\newtheorem{cor}[thm]{Corollary}
\newtheorem{prop}[thm]{Proposition}
\newtheorem{lem}[thm]{Lemma}

\theoremstyle{definition}
\newtheorem{defn}[thm]{Definition}

\theoremstyle{remark}

\makeatletter
\let\c@equation\c@thm
\makeatother
\numberwithin{equation}{section}

\DeclareMathOperator{\Li}{Li}

\title{Cancellations in power series of sine type}
\author{J. Arias de Reyna}
\address{Facultad de Matem\'aticas \\
Univ.~de Sevilla \\
Apdo.~1160
 \\
41080-Sevilla \\
Spain} 
\thanks{Supported by  MINECO grant MTM2012-30748.}
\email{arias@us.es}

\date{\today}

\begin{document}

\newcommand{\N}{{\mathbb N}}
\newcommand{\R}{{\mathbb R}}
\newcommand{\C}{{\mathbb C}}
\newcommand{\Z}{{\mathbb Z}}
\newcommand{\Q}{{\mathbb Q}}
\newcommand{\arctanh}{\mathop{\rm arctanh }}
\def\Re{\operatorname{Re}}
\def\Im{\operatorname{Im}}
\newfont{\cmbsy}{cmbsy10}
\newcommand{\Orden}{\mathop{\hbox{\cmbsy O}}\nolimits}

\begin{abstract} 
We present a method to study the behavior of a power series of type
\[f(x):=\sum_{n=0}^\infty (-1)^n c_n\frac{x^{2n+1}}{(2n+1)!}\]
when $x\to\infty$. 

We apply our method to study the function
\[f(t):=\int_0^t\frac{dx}{x}\int_0^x\frac{dy}{y}\int_0^y\frac{dz}{z}\bigl\{
\sin x+\sin(x-y)-\sin(x-z)-\sin(x-y+z)\bigr\}.\]
We will derive various different representations of $f(t)$
by means of which it will be shown that $\lim_{t\to+\infty}f(t)=0$, 
disproving  a conjecture by Z. Silagadze, claiming that this limit equals $-\pi^3/12$.

\end{abstract}

\maketitle

\section{Introduction.}
In  Titchmarsh  \cite[Section 14.32]{T} we find two suggestive
equivalents to the Riemann Hypothesis: The RH is equivalent to 
$F(x)=\Orden(x^{\frac12+\varepsilon})$,  or alternatively to $G(x)=\Orden(x^{-\frac14+\varepsilon})$
where
\begin{equation}\label{RieszF}
F(x)=\sum_{n=1}^\infty\frac{(-1)^{n+1}x^n}{(n-1)!\zeta(2n)},\quad
G(x)=\sum_{n=1}^\infty\frac{(-1)^n x^n}{n!\zeta(2n+1)}.
\end{equation}
The first equivalence is due to Riesz \cite{R}, and the second to Hardy and Littlewood \cite{HL}. 

Since $\zeta(n)$ converges to $1$ these series can be considered slight modifications of the 
exponential function. As in the case of the exponential the series are convergent
everywhere but
the small values they get for $x$ large is the result of an amazing cancellation  between large terms of
different signs.  This phenomenon happens also in the case of the sine or cosine 
series. In these simple cases the many algebraic properties of the
corresponding sums yield the proof of 
the cancellations. But how can one treat a case as the series \eqref{RieszF} above?

A  problem in MathOverflow leads us to consider the power series
\begin{equation}\label{power}
f(t)=\sum_{n=1}^\infty(-1)^n
\Bigl(\sum_{k=1}^{2n+1}\frac{H_{k-1}}{k}\Bigr)\frac{t^{2n+1}}{(2n+1)!(2n+1)}
\end{equation}
where the $H_n=\sum_{k=1}^n \frac1k$ are the harmonic numbers. The proposer Z. Silagadze 
asked for a proof that $f(t)\to-\pi^3/12$ (obtained by using some
arguments from physics). 

We  think that our solution presented here is interesting because 
it provides an example of how to treat this type of problems.  It must be said that
the method can be applied to the functions in \eqref{RieszF}, but as we will see 
in Section \ref{S:Riesz}
it only gives another path to the connection between these series and the Riemann Hypothesis.

Silagadze's problem in MathOverflow \cite{S2} was to compute
\begin{equation}\label{E:integral}
\int_0^{\infty}\frac{dx}{x}\int_0^x\frac{dy}{y}\int_0^y\frac{dz}{z}\bigl\{
\sin x+\sin(x-y)-\sin(x-z)-\sin(x-y+z)\bigr\}
\end{equation}
for which he conjectured the value $-\pi^3/12$.

This integral is not absolutely convergent. Being a multiple integral 
it is not clear in what sense he is asking  to compute it. 
The most natural interpretation of \eqref{E:integral} is to define for  $t>0$ 
\begin{equation}\label{def f}
f(t):=\int_0^t\frac{dx}{x}\int_0^x\frac{dy}{y}\int_0^y\frac{dz}{z}\bigl\{
\sin x+\sin(x-y)-\sin(x-z)-\sin(x-y+z)\bigr\}.
\end{equation}
and define the integral \eqref{E:integral} as the 
$\lim_{t\to\infty} f(t)$.
We will prove that this limit $=0$. 

The function $f(t)$ defined by the triple integral in \eqref{def f}
has many interesting properties. 
It extends to  an entire function
whose power series \eqref{power} is a slight modification of the series for $\sin t$. 
We are interested in its behavior for $t\to+\infty$. This is a similar problem
as the one (equivalent to the Riemann Hypothesis) introduced by M. Riesz cited above.  
We will solve the problem in this relatively simple case, essentially, by means of a 
representation of 
$f(t)$ as a Fourier transform \eqref{E:third}. 

The transformations  we apply to the function $f(t)$ conclude with the representation
as a Fourier transform in Proposition \ref{P:Fourier}. But all our representations
are needed to reach this final one which will solve Silagadze's problem.

Our formulas \eqref{E:series} and \eqref{E:third} may also be applied to the computation of the triple integral.
The numerical computation of a triple integral is quite often difficult. By means of the 
power series or the Fourier representation it can be computed easily and above all 
more reliably for relatively small values of $t$.  For large values of $t$ we present explicitly an asymptotic expansion (Section \ref{S:asymp}) that is 
very well suited to compute $f(t)$ with high precision for $t$ large. 
The first term of the asymptotic expansion 
\[f(t)= -\frac{\cos t}{2}\frac{\log^2t}{t}+\Orden\Bigl(\frac{\log t}{t}\Bigr)\]
shows that $f(t)$ has zeros near the zeros of $\cos t$ for $t$ large.

As an application we will also obtain the x-ray of $f(t)$. It shows 
that  the zeros with small absolute value are real.  On average there are two zeros
on each interval of length $2\pi$. The first zeros are  separated by approximately 
4, 2, 4, 2, \dots   The x-ray also shows that the general properties of $f(t)$
are very similar to those of $\sin t$.

To finish the paper we give some details about the Riesz function analogous to $f(t)$, 
also explaining why we may not hope that an analogous analysis will solve the RH. As expected!

\section{The entire function $f(t)$.}

\begin{prop}
The function $f(t)$ defined for $t>0$ by the absolutely convergent integral
\begin{equation}\label{E:fdef}
f(t):=\int_0^t\frac{dx}{x}\int_0^x\frac{dy}{y}\int_0^y\frac{dz}{z}\bigl\{
\sin x+\sin(x-y)-\sin(x-z)-\sin(x-y+z)\bigr\}
\end{equation}
extends to an entire function with power series expansion
\begin{equation}\label{E:series}
f(t)=\sum_{n=1}^\infty(-1)^n\Bigl(
\sum_{k=1}^{2n+1}\frac{H_{k-1}}{k}\Bigr)\frac{t^{2n+1}}{(2n+1)!(2n+1)}
\end{equation}
where $H_n=\sum_{k=1}^n\frac{1}{k}$ is the $n$-th harmonic number.
\end{prop}

\begin{proof}
Applying the Mean Value Theorem
we get the following bound of the absolute value of the integrand in the definition of $f(t)$
\[\frac{1}{x}\cdot\frac{1}{y}\Bigl|\frac{\sin x-\sin(x-z)}{z}-\frac{\sin(x-y+z)-\sin(x-y)}
{z}\Bigr|\le \frac{2}{xy}.\]
When integrating over $z\in(0,y)$ we get something bounded by $2/x$, and integrating this 
over $y\in(0,x)$ we get something bounded by $2$. The integral of
this over $x\in(0,t)$ yields something bounded by $2t$. This shows that
the integral in \eqref{E:fdef} is absolutely convergent.

Changing variables  $x=tu$, $y=tv$, $z=tw$ yields
\[f(t)=\int_0^1\frac{du}{u}\int_0^u\frac{dv}{v}\int_0^v\frac{dw}{w}\bigl\{
\sin (tu)+\sin(tu-tv)-\sin(tu-tw)-\sin(tu-tv+tw)\bigr\}\]
and renaming the  variables
\[f(t)=\int_0^1\frac{dx}{x}\int_0^x\frac{dy}{y}\int_0^y\frac{dz}{z}\bigl\{
\sin (tx)+\sin(tx-ty)-\sin(tx-tz)-\sin(tx-ty+tz)\bigr\}.\]
This representation shows that $f(t)$ extends to an entire function with power series 
expansion
\[f(t)=\sum_{n=0}^\infty(-1)^n
A_n\frac{t^{2n+1}}{(2n+1)!}\]
where
\[A_n=\int_0^1\frac{dx}{x}\int_0^x\frac{dy}{y}\int_0^y\frac{dz}{z}\bigl\{x^{2n+1}+(x-y)^{2n+1}-(x-z)^{2n+1}-(x-y+z)^{2n+1}\}.\]
Notice that $A_0=0$. 
To compute the other $A_n$ we get successively
\begin{multline*}
\int_0^y\bigl\{x^{2n+1}-(x-z)^{2n+1}\}\frac{dz}{z}=
-\sum_{k=1}^{2n+1}(-1)^{k}\binom{2n+1}{k}\frac{y^k}{k}x^{2n+1-k},\\
\int_0^x\frac{dy}{y}\int_0^y\bigl\{x^{2n+1}-(x-z)^{2n+1}\}\frac{dz}{z}=
-\sum_{k=1}^{2n+1}(-1)^{k}\binom{2n+1}{k}\frac{x^{2n+1}}{k^2},\\
\int_0^1\frac{dx}{x}\int_0^x\frac{dy}{y}\int_0^y\bigl\{x^{2n+1}-(x-z)^{2n+1}\}\frac{dz}{z}=-\frac{1}{2n+1}\sum_{k=1}^{2n+1}(-1)^{k}\binom{2n+1}{k}\frac{1}{k^2}\\
=\frac{1}{2n+1}
\sum_{k=1}^{2n+1}\frac{H_k}{k}
\end{multline*}
where the last step is proved in Lemma \ref{L:}.

For the other part of $A_n$ we compute
\begin{multline*}
\int_0^y\bigl\{(x-y)^{2n+1}-(x-y+z)^{2n+1}\}\frac{dz}{z}=
-\sum_{k=1}^{2n+1}\binom{2n+1}{k}\frac{y^k}{k}(x-y)^{2n+1-k},\\
\int_0^x\frac{dy}{y}\int_0^y\bigl\{(x-y)^{2n+1}-(x-y+z)^{2n+1}\}\frac{dz}{z}\\
=-\sum_{k=1}^{2n+1}\frac{1}{k}\binom{2n+1}{k}\int_0^xy^k(x-y)^{2n+1-k}\frac{dy}{y}.
\end{multline*}
In the last integral we change variables $y=xu$ yielding
\begin{multline*}
=-\sum_{k=1}^{2n+1}\frac{x^{2n+1}}{k}\binom{2n+1}{k}\int_0^1 u^k(1-u)^{2n+1-k}
\frac{du}{u}=\\
=-\sum_{k=1}^{2n+1}\frac{x^{2n+1}}{k}\frac{(2n+1)!}{k!(2n+1-k)!}\frac{(k-1)!(2n+1-k)!}
{(2n+1)!}=-\sum_{k=1}^{2n+1}\frac{x^{2n+1}}{k^2}
\end{multline*}
so that 
\[\int_0^1\frac{dx}{x}\int_0^x\frac{dy}{y}\int_0^y\bigl\{(x-y)^{2n+1}-(x-y+z)^{2n+1}\}\frac{dz}{z}=-\frac{1}{2n+1}\sum_{k=1}^{2n+1}\frac{1}{k^2}.\]
Collecting all our partial results we obtain
\[A_n=\frac{1}{2n+1}
\sum_{k=1}^{2n+1}\frac{H_k}{k}-\frac{1}{2n+1}\sum_{k=1}^{2n+1}\frac{1}{k^2}=
\frac{1}{2n+1}
\sum_{k=1}^{2n+1}\frac{H_{k-1}}{k}.\]
\end{proof}

\begin{lem}\label{L:}
For any natural number $n$  we have
\begin{equation}
\sum_{k=1}^{n}(-1)^{k}\binom{n}{k}\frac{1}{k^2}=-
\sum_{k=1}^{n}\frac{H_k}{k}.
\end{equation}
\end{lem}

\begin{proof}
To prove $a_n=b_n$ it is sufficient to prove $a_1=b_1$ and $a_n-a_{n-1}=b_n-b_{n-1}$.
In our case the equality for $n=1$ is checked easily. 
So, we have to show
\[\sum_{k=1}^n(-1)^{k+1}\binom{n}{k}\frac{1}{k^2}-
\sum_{k=1}^{n-1}(-1)^{k+1}\binom{n-1}{k}\frac{1}{k^2}=\frac{H_n}{n}.\]
Since $\binom{n}{k}=\binom{n-1}{k-1}+\binom{n-1}{k}$, 
the equality is equivalent to 
\[\frac{(-1)^{n+1}}{n^2}+\sum_{k=1}^{n-1}(-1)^{k+1}\binom{n-1}{k-1}\frac{1}{k^2}=\frac{H_n}{n}\]
or, multiplying by $n$
\begin{equation}\label{E:GR}
\sum_{k=1}^{n}(-1)^{k+1}\binom{n}{k}\frac{1}{k}=H_n.
\end{equation}
This is formula 0.155.4 in Gradshteyn and Ryzhik \cite[p.~4]{GR}.
\end{proof}

\section{Plot of the function $f(t)$.}
The power series expansion of $f(t)$ is analogous to the one for  $\sin t$. 
These power series are not well suited for computation, 
because of the violent  cancellation
between large terms.  Nevertheless, one may use it using high
precision in the computation of the terms to get approximate values. 
Assuming $t>0$, each term of the power series \eqref{E:series} is in absolute value
less than $t^{2n+1}/(2n+1)!$. Therefore the error committed by summing only the 
first $N>3t$ terms of the series \eqref{E:series} is less than $2^{-2N}$. 
It follows that to compute $f(t)$ with error less than $2\varepsilon$ we need
only compute the sum of the first $N$ terms of the series with error less 
than $\varepsilon$,  taking $N$ large enough  so that 
\[N\ge \frac{\log(1/\varepsilon)}{2\log2},\quad N>3t.\]
All the terms of the series are less than $t^{2n+1}/(2n+1)! <(\frac{et}{2n+1})^{2n+1}\le e^t$. So, we must compute each term with an error less than $\varepsilon/N$, for which
it will suffice  to compute each term working with a precision
\[P:=\frac{t+\log(N/e)}{\log 10} \quad \text{decimal digits}.\]
Of course this will be difficult for $t$ very large, but today we may easily compute 
with thousands of digits of precision.  

\bigskip

\includegraphics[width=0.95\linewidth]{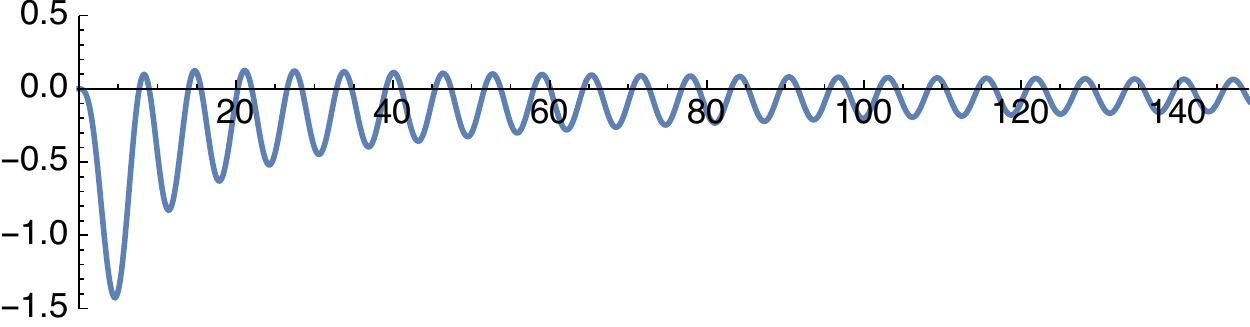}

We may say that the computation of the power series is much easier than the 
computation of the multiple integral \eqref{E:fdef} or any other multiple integral 
giving $f(t)$ considered in this paper. This is true even when $t$ is small.

\section{Computation of the limit of $f(t)$.}

First we prove the following integral representation of $f(t)$:
\begin{prop}
For $t>0$ we have
\begin{multline}\label{E:second}
f(t)=\\\int_0^tdu\int_0^1\int_0^1\Bigl(\frac{x\sin u}{u(1-x)(1-x y)}-\frac{\sin(ux)}{u(1-x)(1-y)}+\frac{\sin(uxy)}{u(1-y)(1-xy)}\Bigr)\,dx\,dy.
\end{multline}
\end{prop}

\begin{proof}
Starting from the power series expansion it is easy to get
\begin{multline*}
f(t)=\sum_{n=1}^\infty(-1)^n
\sum_{k=1}^{2n+1}\frac{1}{k}\sum_{j=1}^{k-1}\frac{1}{j}\frac{t^{2n+1}}{(2n+1)!(2n+1)}\\
=\int_0^tdu\int_0^1dx\int_0^1dy\sum_{n=1}^\infty(-1)^n
\sum_{k=1}^{2n+1}x^{k-1}\sum_{j=1}^{k-1}y^{j-1}\frac{u^{2n}}{(2n+1)!}.
\end{multline*}
If  $C$ is a circle with radius $r>t$, we may express the factorial by means of 
the Residues Theorem
\begin{equation}\label{E:absol}
f(t)=
\int_0^tdu\int_0^1dx\int_0^1dy \frac{1}{2\pi i}\int_C\frac{e^z}{z^2}\sum_{n=1}^\infty(-1)^n\sum_{k=1}^{2n+1}x^{k-1}\sum_{j=1}^{k-1}y^{j-1}(u/z)^{2n}\,dz.
\end{equation}
(Since $|u/z|<1$ all these series converge absolutely and the order of the integrals and
summations may be interchanged). 
This can be written as
\[f(t)=\
\int_0^tdu\int_0^1dx\int_0^1dy \frac{1}{2\pi i}\int_C\frac{e^z}{z^2}\frac{u^4x^3yz^2-u^2xz^4-u^2x^2z^4-u^2x^2yz^4}{(u^2+z^2)(u^2x^2+z^2)(u^2x^2y^2+z^2)}\,dz.\]
By the Residue Theorem  we obtain 
\begin{multline*}
\frac{1}{2\pi i}\int_C\frac{e^z}{z^2}\frac{u^4x^3yz^2-u^2xz^4-u^2x^2z^4-u^2x^2yz^4}{(u^2+z^2)(u^2x^2+z^2)(u^2x^2y^2+z^2)}\,dz\\
=\frac{x\sin u}{u(1-x)(1-x y)}-\frac{\sin(ux)}{u(1-x)(1-y)}+\frac{\sin(uxy)}{u(1-y)(1-xy)}
\end{multline*}
establishing our claim \eqref{E:second}.  Observe that since our integral representation 
\eqref{E:second} appears
after computing one of the integrals in  the absolutely 
convergent integral \eqref{E:absol},
the integrals in our new representation are also absolutely convergent. 
\end{proof}

Notice that in the representation \eqref{E:second} we may interchange the integral 
in $u$ with the integral in $(x,y)$, but then  it is not easy  to justify the interchange 
of the limit in $t$ with the integrals,
because the integrand is not dominated by an integrable function. If we would 
proceed formally
in this way we would easily get that the limit is $0$, but, as  said before, this is not allowed. 
Therefore we follow another path.

\begin{prop}\label{P:Fourier}
For any complex $t$ we have
\begin{equation}\label{E:third}
f(t)=\int_0^1\Bigl(\frac12\log^2(1-x)+\sum_{n=1}^\infty\frac{(1-x)^n-1}{n^2}\Bigr)
\frac{\sin(tx)}{x}\,dx.
\end{equation}
\end{prop}

\begin{proof}
We take $t>0$ in the representation \eqref{E:second}, and change the order of integration
\begin{multline*}
f(t)=\\\int_0^1\int_0^1\,dx\,dy\int_0^t\Bigl(\frac{x\sin u}{u(1-x)(1-x y)}-\frac{\sin(ux)}{u(1-x)(1-y)}+\frac{\sin(uxy)}{u(1-y)(1-xy)}\Bigr)\,du.
\end{multline*}
Now subdivide the inner integral in three and  change variables appropriately 
\begin{multline*}
f(t)=\int_0^1\int_0^1\,dx\,dy\Bigl(\int_0^t\frac{x\sin u}{u(1-x)(1-x y)}\,du\\
-
\int_0^{tx}\frac{\sin u}{u(1-x)(1-y)}\,du
+\int_0^{txy}\frac{\sin u}{u(1-y)(1-xy)}\,du\Bigr).
\end{multline*}
We use Iverson's notation \cite{K}, so that, for any proposition $P$, the symbol
$[P]$ is $1$ if $P$ is true and $0$ if
it is false. In this way we may write
\begin{multline*}
f(t)=\\\int_0^1\int_0^1\,dx\,dy\int_0^t\Bigl(\frac{x[u<t]\sin u}{u(1-x)(1-x y)}\,du-
\frac{[u<tx]\sin u}{u(1-x)(1-y)}+\frac{[u<txy]\sin u}{u(1-y)(1-xy)}\Bigr)\,du.
\end{multline*}
To simplify the notation we put $a:=u/t$,  so that  always $0<a<1$. Interchanging the integrals we obtain
\begin{multline*}
f(t)=\\ \int_0^t\frac{\sin u}{u}\,du\int_0^1\int_0^1\Bigl(\frac{x}{(1-x)(1-x y)}\,du-
\frac{[a<x]}{(1-x)(1-y)}+\frac{[a<xy]}{(1-y)(1-xy)}\Bigr)\,dx\,dy
\end{multline*}
and will compute the inner double integral in $x$ and $y$. This is 
\[J(a):=\int_0^1\int_0^1\Bigl(\frac{x}{(1-x)(1-x y)}\,du-
\frac{[a<x]}{(1-x)(1-y)}+\frac{[a<xy]}{(1-y)(1-xy)}\Bigr)\,dx\,dy.\]
We subdivide the square $(0,1)^2$ into three disjoint sets
\begin{multline*}
S_1:=\{(x,y): 0<x\le a, 0<y<1\},\quad
S_2:=\{(x,y): a<x<1, 0<xy\le a\},\quad \\
S_3:=\{(x,y): a<xy\}.
\end{multline*}
In $S_1$ the integrand is $=\frac{x}{(1-x)(1-x y)}$,
and in $S_2$ it is 
\[=\frac{x}{(1-x)(1-x y)}-\frac{1}{(1-x)(1-y)}=-\frac{1}{(1-y)(1-xy)}.\]
In $S_3$ the integrand is equal to 
\[=\frac{x}{(1-x)(1-x y)}-\frac{1}{(1-x)(1-y)}+\frac{1}{(1-y)(1-xy)}=0\]
so that 
\begin{multline*}
J(a)=\int_0^1\,dy\int_0^a\frac{x}{(1-x)(1-x y)}\,dx-
\int_a^1\,dx\int_0^{a/x}\frac{1}{(1-y)(1-xy)}\,dy\\
=\frac{1}{2}\log^2(1-a)-\frac{\pi^2}{6}+
\sum_{n=1}^\infty \frac{(1-a)^n}{n^2}.
\end{multline*}
This yields
\[f(t)=\int_0^t \Bigl(\frac{1}{2}\log^2(1-u/t)+
\sum_{n=1}^\infty \frac{(1-u/t)^n-1}{n^2}\Bigr)\frac{\sin u}{u}\,du.\]
A  change of variables $u=tx$ yields \eqref{E:third} for $t>0$, and it is clear
that the two sides are entire functions.
\end{proof}

\begin{prop}
We have
\[\lim_{t\to\infty} f(t)=\lim_{t\to+\infty}\int_0^1\Bigl(\frac12\log^2(1-x)+\sum_{n=1}^\infty\frac{(1-x)^n-1}{n^2}\Bigr)
\frac{\sin(tx)}{x}\,dx=0.\]
\end{prop}

\begin{proof}
This is an example of the Riemann-Lebesgue Lemma, once it has been shown 
that the function 
\[g(x):=\frac{1}{x}\Bigl(\frac12\log^2(1-x)+\sum_{n=1}^\infty\frac{(1-x)^n-1}{n^2}\Bigr)\]
is in ${\mathcal L}^1[0,1]$. 
This is a simple exercise.
\end{proof}

\section{Asymptotic Expansion.}\label{S:asymp}

The representation \eqref{P:Fourier} as a Fourier integral yields by known methods  
an asymptotic expansion for $f(t)$. First notice that the dilogarithm function
$\Li_2(z)$ can be defined in the cut plane $\C\smallsetminus[1,\infty)$ by
\begin{equation}
\Li_2(z)=-\int_0^z\log(1-u)\frac{du}{u}
\end{equation} 
where the path of integration is the segment joining $0$ and $z$.  For $z$ in the plane 
with two cuts along the real axis $(-\infty,0]$ and $[1,+\infty)$ the dilogarithm satisfies
the Euler functional equation (cf.~Lewin \cite[(1.12), p.~5]{Le})
\begin{equation}
\Li_2(z)+\Li_2(1-z)=\frac{\pi^2}{6}-\log z\log(1-z).
\end{equation}
Here and in the sequel we will denote by  $\log w$  the principal branch of the logarithm. 

\begin{defn}
Let $\Omega$ be the complex plane with the two cuts $(-\infty,0]$ and $[1,+\infty)$. 
We define the function $g(z)$ holomorphic in $\Omega$ by 
\begin{multline}\label{E:gg}
g(z):=\Bigl(\frac12\log^2(1-z)-\frac{\pi^2}{6}+\Li_2(1-z)\Bigr)
\frac{1}{z}\\=\Bigl(\frac12\log^2(1-z)-\Li_2(z)-\log z\log(1-z)\Bigr)
\frac{1}{z}.
\end{multline}
\end{defn}

It is easier to give explicitly the asymptotic expansion of the Fourier 
transform
\begin{equation}
J(t):=\int_0^1 g(x)e^{ixt}\,dx.
\end{equation}
Since $f(t)=\Im J(t)$ its asymptotic expansion can be obtained 
easily from the one of $J(t)$.

\begin{prop}\label{P:Laplace}
For any $t>0$ we have $f(t)=\Im J(t)$ where
\begin{equation}\label{E:J}
J(t)=i\int_0^\infty g(iy)e^{-ty}\,dy-i e^{it}\int_0^\infty g(1+iy)e^{-ty}\,dy.
\end{equation}
\end{prop}

\begin{proof}
By the definition of $g(z)$ and \eqref{P:Fourier} we have
\[f(t)=\Im\int_0^1 g(x)e^{ixt}\,dx.\]
We apply Cauchy's Theorem to a rectangle with vertices at $0$, $1$, 
$1+iR$, $iR$ with $R>0$
and let $R\to+\infty$. In this way the integral in $(0,1)$ can be converted  into the 
two infinite integrals in the Proposition.  The bounds are easy. 
\end{proof}

Since the two integrals in Proposition \ref{P:Laplace} are Laplace integrals we may apply 
Watson's Lemma \cite[17.03, p. 501]{J} to get their asymptotic expansions. In this case we have the additional difficulty 
of logarithmic singularities at the extremes of the integrals at $0$ and $1$. We follow the path
in Lyness \cite{L}.  To get the asymptotic expansion we need the behavior of $g(z)$ near 
$z=0$ and $z=1$.

\begin{lem}\label{defconstants}
For $|z|<1$ in $\Omega$ we have
\begin{multline}\label{E:near0}
g(z)=\sum_{n=0}^\infty \Bigl(\frac{H_{n+1}}{n+1}-\frac{2}{(n+1)^2}\Bigr)z^n + \Bigl(\sum_{n=0}^\infty \frac{z^n}{n+1}\Bigr)\log z\\
:=\sum_{n=0}^\infty A_nz^n+
\Bigl(\sum_{n=0}^\infty B_nz^n\Bigr)\log z.
\end{multline}
For $z\in\Omega$, with  $|1-z|<1$ we have
\begin{multline}\label{E:near1}
g(z)=-\sum_{n=0}^\infty \Psi'(n+1)(1-z)^n  + \frac12\Bigl(\sum_{n=0}^\infty(1-z)^n\Bigr)\log ^2(1-z)\\:=\sum_{n=0}^\infty C_n(1-z)^n+\Bigl(\sum_{n=0}^\infty D_n(1-z)^n\Bigr)
\log^2(1- z)
\end{multline}
where $\Psi(z)=\Gamma'(z)/\Gamma(z)$ is the digamma function. 
\end{lem}

\begin{proof}
Near $z=0$ we have $g(z)=g_1(z)+g_2(z)\log z$ where  by \eqref{E:gg} we have
\[g_1(z)=\Bigl(\frac12\log^2(1-z)-\Li_2(z)\Bigr)
\frac{1}{z},\quad g_2(z)=-\frac{\log(1-z)}{z}.\]
$g_1(z)$ and $g_2(z)$ are holomorphic at $z=0$. 
Expanding in power series at $z=0$ we get \eqref{E:near0}.

Near $z=1$ we have $g(z)=g_3(z)+g_4(z)\log^2(1-z)$ with $g_3(z)$ and $g_4(z)$ holomorphic at $z=1$. By \eqref{E:gg}
\[g_3(z)=-\frac{\pi^2}{6z}+\frac{\Li_2(1-z)}{z},\quad g_4(z)=\frac{1}{2z}.\]
The functions $g_3(z)$ and $g_4(z)$ are holomorphic at $z=1$ and expanding them in 
power series we obtain \eqref{E:near1}.
We notice the  equalities.
\begin{equation}
\Psi(n+1)= H_n-\gamma ,\quad \Psi'(n+1)=\frac{\pi^2}{6}-\sum_{k=1}^n\frac{1}{k^2}.
\end{equation}
\end{proof}

\begin{prop}
The following asymptotic expansion is valid for $t\to+\infty$
\begin{multline}\label{asymptotic}
J(t)\sim \sum_{n=0}^\infty i^{n+1} A_n\frac{n!}{t^{n+1}}+
\sum_{n=0}^\infty i^{n+1} B_n \Bigl(\Psi(n+1)-\log t+\frac{\pi i}{2}\Bigr)\frac{n!}{t^{n+1}}\\
-i e^{it}\sum_{n=0}^\infty (-i)^n C_n \frac{n!}{t^{n+1}}
-i e^{it}\sum_{n=0}^\infty (-i)^n D_n
\Bigl\{\Psi'(n+1)+\Bigl(\gamma-H_n+\log t+\frac{\pi i}{2}\Bigr)^2 \Bigr\} \frac{n!}{t^{n+1}}
\end{multline}
where the coefficients $A_n$, $B_n$, $C_n$ and $D_n$ are defined in Lemma \ref{defconstants}.
\end{prop}

\begin{proof}
By  the results in Lyness \cite{L} the asymptotic expansion of $J(t)$ is obtained by
integrating term by term the  expression \eqref{E:J} after substituting the 
power series at the extremes at $0$ and $1$, respectively. Explicitly
\begin{multline*}
J(t)\sim \sum_{n=0}^\infty i^{n+1} A_n\int_0^\infty y^n e^{-ty}\,dy+
\sum_{n=0}^\infty i^{n+1} B_n\int_0^\infty y^n\Bigl(\log y+\frac{\pi i}{2}
\Bigr) e^{-ty}\,dy\\
-i e^{it}\sum_{n=0}^\infty (-i)^n C_n \int_0^\infty y^n e^{-ty}\,dy
-i e^{it}\sum_{n=0}^\infty (-i)^n D_n\int_0^\infty y^n
\Bigl(\log y-\frac{\pi i}{2}\Bigr)^2 e^{-ty}\,dy
\end{multline*}
Computing the integrals we obtain \eqref{asymptotic}.
\end{proof}

\begin{cor}
The first order terms of the asymptotic  expansion  of $f(t)$ are 
\begin{equation}
-\frac{\cos t}{2}\frac{\log^2t}{t}+
\Bigl(\frac{\pi\sin t}{2}-\gamma\cos t -1\Bigr)\frac{\log t}{t}+
\Bigl(\frac{5\pi^2\cos t}{24}-\frac{\gamma^2\cos t}{2}+\frac{\gamma\pi\sin t}{2}-1-\gamma\Bigr)\frac{1}{t}
\end{equation}
where $\gamma$ is Euler's constant.
\end{cor}

\section{Riesz type functions}\label{S:Riesz}

We have proved that the function $f(t)$ given by the power series 
\eqref{power} tends to $0$ by writing it as a Fourier transform of an $L^1$ function. 
As an example of what happens in the case of a Riesz-type function we just consider 
Riesz's $F(x)$. Riesz proves \cite{R}
\begin{equation}
F(x)=\sum_{n=1}^\infty\frac{(-1)^{n+1}x^n}{(n-1)!\zeta(2n)}=
x\sum\frac{\mu(n)}{n^2}e^{-\frac{x}{n^2}}.
\end{equation}

The Riemann Hypothesis is equivalent to $F(x)=\Orden(x^{\frac14+\varepsilon})$
for any $\varepsilon>0$. Riesz shows that $F(x)$ is not $\Orden(x^\alpha)$ for any $\alpha<1/4$. 
Therefore, it is not true that $F(x)$ converges to $0$ when $x\to\infty$. 
Riesz proves that $F(x)x^{-1/2}$ converges to $0$.  With some work we may reprove this.
Noticing that 
\begin{equation}
\sqrt{t}\;e^{-t}=\frac{1}{\sqrt{\pi}}\int_0^{+\infty} \frac{\cos(\frac32\arctan x)}{(1+x^2)^{3/4}}
\cos(x t)\,dx,\qquad t>0
\end{equation}
we can prove
\begin{equation}
\frac{F(x)}{\sqrt{x}}=\frac{1}{\sqrt{\pi}}\int_0^{+\infty}\Bigl(
\sum_{n=1}^\infty \mu(n)\frac{n\cos(\frac32\arctan(n^2t))}{(1+n^4t^2)^{3/4}}\Bigr)\cos(xt)\,dt.
\end{equation}
We apply summation by parts and use  $M(x):=\sum_{n\le x}\mu(n)=\Orden(xe^{-c\sqrt{\log x}})$ to show 
that 
\begin{equation}
g(t):=\sum_{n=1}^\infty \mu(n)\frac{n\cos(\frac32\arctan(n^2t))}{(1+n^4t^2)^{3/4}}
\end{equation} 
is in $L^1(0,\infty)$. This implies that $F(x)x^{-1/2}=o(1)$. 

We do not give the details of the proofs, but they are clearly more complicated than
those  given by Riesz. 

We can also obtain in this way that $M(x)=\Orden(x^a)$ with $\frac12\le a<1$ would 
imply that 
$\zeta(s)$ does not vanish on $\Re s>a$. Again quite a difficult way to get this simple result.

\section{The x-ray of the function $f(t)$.}

\includegraphics[width=0.95\linewidth]{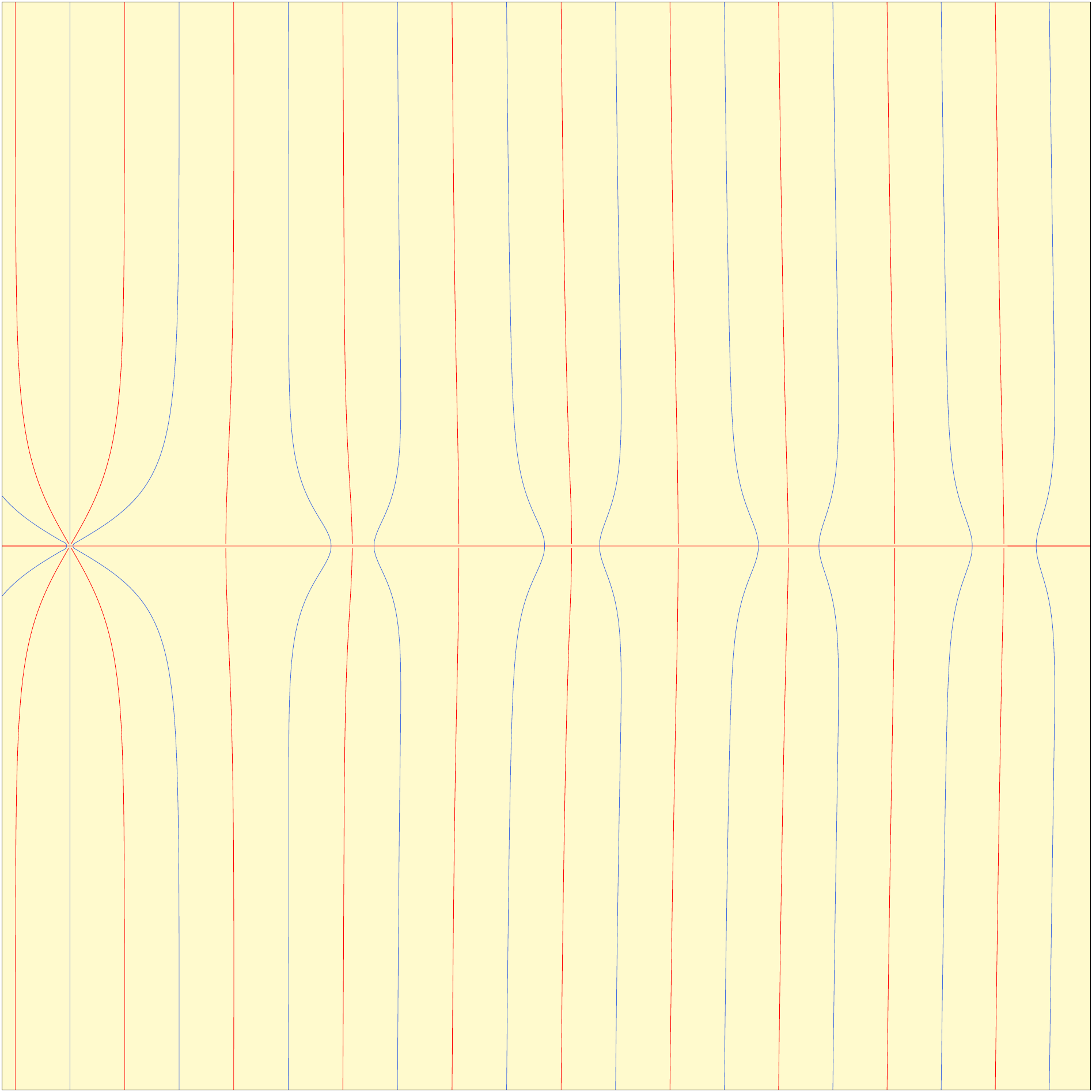}

\medskip

This represents  $f(z)$  on the square $(-2,30)\times(-16,16)$. In red lines
where $f(z)$ takes real values, 
in blue lines where $f(z)$ is purely imaginary. At $t=0$ we observe
a triple zero. Table 1 contains some more zeros of $f(t)$.

\begin{table}[htdp]
\caption{First zeros of $f(t)$.}
\begin{center}
\begin{tabular}{|c|c||c|c||c|c|}
\hline
1 & 7.67705050991057 & 16 & 53.7503680917393 & 31 & 101.945883830959\\
2 & 8.93841793140833 & 17 & 57.9542719238923 & 32 & 104.179794651601\\
3 & 13.9549229413771 & 18 & 60.0643806701981 & 33 & 108.230057788273\\
4 & 15.5679745247884 & 19 & 64.2391371546902 & 34 & 110.475488483408\\
5 & 20.2405336746206 & 20 & 66.3740436665322 & 35 & 114.514173951783\\
6 & 22.0185470304095 & 21 & 70.5238571367734 & 36 & 116.770227743121\\
7 & 26.5267465905312 & 22 & 72.6803262355070  & 37 & 120.798238189933\\
8 & 28.4064753351502 & 23 & 76.8084524955294 & 38 & 123.064133569705\\
9 & 32.8127796622118 & 24 & 78.9839168600548 & 39 & 127.082255551763\\
10 & 34.7636137360244 & 25 & 83.0929401032590 & 40 & 129.357306301363\\
11 & 39.0985258618221 & 26 & 85.2853203526836 & 41 & 133.366230408727\\
12 & 41.1028595266668 & 27 & 89.3773338407949 & 42 & 135.649829884233\\
13 & 45.3839993880412 & 28 & 91.5849167023848 & 43 & 139.650166567860\\
14 & 47.4305876227269 & 29 & 95.6616452128959 & 44 & 141.941775186108\\
15 & 51.6692361998946 & 30 & 97.8829983472283 & 45 & 145.934067362779\\
\hline
\end{tabular}
\end{center}
\label{First zeros of $f(t)$.}
\end{table}%

\section{Acknowledgement.}
I wish to express my thanks to 
Jan van de Lune  ( Hallum, The Netherlands ) for his encouragements and linguistic assistance.

This work has been  supported by  MINECO grant MTM2012-30748.

\end{document}